\documentclass{amsart}

\usepackage{amsmath}
\usepackage{amsfonts}
\usepackage{amssymb}
\usepackage{amsthm}
\usepackage{amscd}

\usepackage{orcidlink}

\usepackage{latexsym}
\usepackage{mathrsfs}
\usepackage{psfrag}
\usepackage[dvips]{epsfig}
\usepackage{epsfig}
\usepackage[all]{xy}        
\theoremstyle{plain}
            
\newtheorem{theorem}{Theorem}[section]
\newtheorem{corollary}[theorem]{Corollary}
\newtheorem{proposition}[theorem]{Proposition}
\newtheorem{lemma}[theorem]{Lemma}

\theoremstyle{definition}
\newtheorem{definition}[theorem]{Definition}

\newtheorem{examples}[theorem]{Examples}

\newtheorem{remark}[theorem]{Remark}

\numberwithin{equation}{section}

\DeclareMathOperator{\Conv}{Conv}
\DeclareMathOperator{\Zero}{Zero}

\begin{document}

\title{Hemiring-valued pseudonormed rings}


\author{Peyman Nasehpour\orcidlink{0000-0001-6625-364X}}

\curraddr{Academic Advisor and Education Mentor\\ Education Department\\ The New York Academy of Sciences\\ New York, NY, USA}


\email{nasehpour@gmail.com}


\dedicatory{}

\begin{abstract}
In the second section, we introduce hemiring-valued pseudonormed rings and generalize Albert's result which states that every finite-dimensional algebra can be normed. Next, we introduce shrinkable hemirings and prove that dense division semirings are shrinkable. In the third section, we show the Cauchy Condensation Test holds for Cauchy complete fields. In the fourth section, we use Bernoulli's inequality to prove a version of ratio test for ring-valued normed groups.  
\end{abstract}

\maketitle

\section{Introduction}

As a generalization to monoid-valued metric spaces (check \cite{Braunfeld2017}, \cite{Conant2019} and Definition 2.10 in \cite{NasehpourParvardi2018}), we introduced magma-valued metric spaces and investigated their properties in \cite{Nasehpour2026}. The current paper continues the author's work on magma-valued metric spaces, although in a distinct context, i.e., hemiring-valued normed rings which generalize both normed rings (see Definition \ref{Hpseudonormedringsdef}) and valuation rings (cf. Examples \ref{examplesofpseudonormedrings}). Before reporting the contributions of this work, we establish some key notations and terminologies.

We recall that $(H,+,\cdot,0, \leq)$ is an ordered hemiring if the following properties are satisfied:

\begin{enumerate}
	\item $(H,+,0,\leq)$ is a commutative ordered monoid;
	\item $(H,\cdot)$ is a semigroup, and $a \leq b$ and $0 \leq c$ imply that $ac \leq bc$ and $ca \leq cb$, for all $a,b,c \in H$.
	\item The multiplication distributes over addition from both sides and $0$ is an absorbing element of $H$, i.e. $0h = h0 = 0$, for all $h\in H$.
\end{enumerate}

A hemiring $H$ is said to be commutative if its multiplication is commutative, i.e. $ab = ba$ for all $a$ and $b$ in $H$. We say $(H,+,\cdot,0, <)$ - for short, $(H,<)$ - is an ordered hemiring if $(H,+,\cdot,0, \leq)$ is an ordered hemiring and $<$ is compatible with addition and $a<b$ and $0<c$ imply that $ac < bc$ and $ca < cb$ for all $a,b,c \in H$. It is evident that if $\leq$ is a total ordering on $H$ and $(H,<)$ is an ordered hemiring, then $H$ is entire, i.e. $ab = 0$ implies that either $a=0$ or $b = 0$, for all $a,b \in H$.

Let $(H,<)$ be an ordered hemiring. We say a ring $R$ is an $H$-pseudonormed ring if there is a function $\Vert \cdot\Vert : R \rightarrow H$ satisfying the following properties:
	\begin{enumerate}
		\item $\Vert r\Vert \geq 0$, $\Vert r\Vert  = 0$ if and only if $r = 0$, for all $r \in R$;
		\item $\Vert r-s\Vert  \leq \Vert r\Vert +\Vert s\Vert $ and $\Vert rs\Vert  \leq \Vert r\Vert \Vert s\Vert $ for all $r,s \in R$.
	\end{enumerate}
	
	In such a case, the function $\Vert \cdot\Vert : R \rightarrow H$ is called an $H$-pseudonorm (see Definition \ref{Hpseudonormedringsdef}). Examples of hemiring-valued pseudonormed rings include real pseudonormed rings \cite{ArnautovGlavatskyMikhalev1996} and valuation rings explained in the classical book \cite{Artin1957} by Artin. 
	
	In \S\ref{sec:hemiringvaluedrings}, we generalize a key result of Albert \cite[Theorem 4]{Albert1947} proving that if $(H,<)$ is a totally ordered commutative hemiring and $F$ is a field equipped with an $H$-pseudonorm, then every finite-dimensional $F$-algebra admits the structure of an $H$-pseudonormed ring (see Theorem \ref{finitedimpseudonormedthm}).

We define an ordered hemiring $H$ to be shrinkable (see Definition \ref{shrinkabledef}) if the following property holds:

\begin{itemize}
	\item For the given positive elements $\alpha$ and $M$ in $H$, there are positive elements $\alpha_l$ and $\alpha_r$ such that \[M \alpha_r < \alpha \text{~and~} \alpha_l M < \alpha.\]
\end{itemize}

We say an ordered hemiring $(H,+,\cdot,<)$ is dense if its additive monoid is dense (see Definition \ref{densehemiringdef}). We recall that $(M,+,0,\leq)$ is a dense ordered unital magma if $+$ is a binary operation on $M$ with the following properties:

\begin{enumerate}
	\item $m+0 = 0+m = m$, for all $m \in M$,
	\item the partial ordering $\leq$ is compatible with the binary operation $+$ on $M$.
	
	\item For any $\epsilon > 0$, there are two elements $\alpha > 0$ and $\beta > 0$ with $\alpha + \beta < \epsilon$ (see Definition 2.1 in \cite{Nasehpour2026}).
\end{enumerate}

Examples of shrinkable hemirings include 

\begin{itemize}
	\item dense ordered division semirings (cf. Proposition \ref{densedivisionsemiringshrinkable}), 
	
	\item DeMarr division rings (cf. Definition 2.12 in \cite{Nasehpour2026} and Examples \ref{shrinkablehemiringsex}) and \item the ring: \[\mathbb Z[1/p] = \{m/p^n: m \in \mathbb Z, n \in \mathbb N\} \qquad\text{(cf.~Proposition~\ref{pnumbersshrinkable})}).\]
\end{itemize}

In Theorem \ref{multiplicationconvergentsequences}, we prove that if $(S,\leq)$ is a dense and shrinkable ring and a join-semilattice, and $R$ is an $S$-pseudonormed ring, then $\Conv(R)$ is a ring and the function \[\varphi: \Conv(R) \rightarrow R \text{~defined~by~} (x_n) \mapsto \lim_{n \to +\infty} x_n\] is a ring epimorphism with $\ker(\varphi) = \Zero(R)$, where by $\Conv(R)$, we mean the set of all convergent sequences in $R$ and, by $\Zero(R)$, the set of all sequences in $R$ convergent to $0 \in R$. In particular in the same result, we show that if $R$ is a commutative ring with identity, then the ideal $\Zero(R)$ is prime (maximal) if and only if $R$ is an integral domain (field). Recall that a ring $R$ is dense if and only if there is a positive element in $R$ smaller than its multiplicative identity $1$ (see Theorem 2.8 in \cite{Nasehpour2026}). 
	
In \S\ref{sec:infiniteseries}, we discuss infinite series in $M$-normed groups and generalize some of the classical results of mathematical analysis discussed in \cite{Ovchinnikov2021}. For example, in Theorem \ref{DenseCauchycompletethm}, we prove that if $(R,<)$ is a dense and Cauchy complete ordered ring and $(x_n)$, $(y_n)$, and $(z_n)$ are sequences in $R$ such that $n \geq N_1$ implies \[x_n \leq y_n \leq z_n,\] and $\sum_{n=1}^{+\infty} x_n$ and $\sum_{n=1}^{+\infty} z_n$ are convergent, then so is the series $\sum_{n=1}^{+\infty} y_n$. We also generalize Cauchy Condensation Test for Cauchy complete fields. In other words, we show that if $\leq$ is a total ordering on $F$ and $(F,<)$ a Cauchy complete ordered field, and $(x_n)$ is a positive and decreasing sequence in $F$, then $\sum_{i=1}^{+\infty} x_i$ is convergent in $F$ if and only if $\sum_{i=0}^{+\infty} 2^i x_{2^i}$ is convergent in $F$ (see Corollary \ref{Cauchycondensationtestcor}).

In \S\ref{sec:geometricseries}, we discuss the convergence of geometric series in dense and shrinkable commutative rings. In fact, in Theorem \ref{Geometricseries}, we prove that if $\leq$ is a total ordering on $R$ and $(R,<)$ an ordered commutative ring with $1$ such that $R$ is a dense and shrinkable ring, then for $r\neq1$, the geometric series $\sum_{n=0}^{+\infty} r^n$ is convergent in $R$ if and only if the sequence $(r^n)$ converges to zero and $1-r$ is invertible in $R$. 

In the same section, we generalize Bernoulli's inequality for ordered semirings (check Theorem \ref{Mitrinovicsemirings}) and prove that if $(S,\leq)$ is an ordered semiring and $(x_i)_{i=1}^{n}$ is a sequence of $n$ elements in $S$ such that $x_i\geq 0$ and $1+x_i \geq 0$, for each $1 \leq i \leq n$, then the following inequality holds: \[\prod_{i=1}^{n} (1+x_i) \geq 1 + \sum_{i=1}^{n} x_i.\] We use a ring version of this inequality to show that if $(F, <)$ is a totally ordered and Archimedean field and $r \in F$ with $- 1 < r < 1$, then $\lim_{n \to +\infty} r^n = 0$ and $\sum_{n=0}^{+\infty} r^n = \displaystyle \frac{1}{1-r}$. (see Proposition \ref{limitpowersequence} and Corollary \ref{GeometricseriesOvchinnikov2021}).

We finalize the paper by proving that if $\leq$ is a total ordering on $R$ and $(R,<)$ is an ordered commutative ring with $1$ such that $R$ is dense and shrinkable; also, if $1-r$ is invertible and $(r^n)$ converges to zero in $R$, and finally, if $(G,+)$ is a Cauchy complete $R$-normed abelian group and $(x_n)_{n=0}^{+\infty}$ is a sequence in $G$ such that $\Vert x_{n+1}\Vert  \leq r\Vert x_n\Vert $ for all $n$, then $\sum_{n=0}^{+\infty} x_n$ is convergent in $G$ (see Theorem \ref{ratiotestcounterpart}). Note that this is a version of the ratio test for ring-valued normed groups.

\section{$H$-pseudonormed rings}\label{sec:hemiringvaluedrings}

\begin{definition}\label{Hpseudonormedringsdef}
Let $R$ be a ring and $(H,\leq)$ an ordered hemiring. We say $R$ is an $H$-pseudonormed ring if there is a function $\Vert \cdot\Vert : R \rightarrow H$ satisfying the following properties:
\begin{enumerate}
	\item $\Vert r\Vert \geq 0$, $\Vert r\Vert  = 0$ if and only if $r = 0$, for all $r \in R$;
	\item $\Vert r-s\Vert  \leq \Vert r\Vert +\Vert s\Vert $ and $\Vert rs\Vert  \leq \Vert r\Vert \Vert s\Vert $ for all $r,s \in R$.
\end{enumerate}

In such a case, the function $\Vert \cdot\Vert : R \rightarrow H$ is called an $H$-pseudonorm. Furthermore, if we have $\Vert rs\Vert  = \Vert r\Vert \Vert s\Vert $, for all $r,s \in R$, we say $R$ is an $H$-normed ring and the function $\Vert \cdot\Vert $ is its $H$-norm.
\end{definition}

\begin{examples}\label{examplesofpseudonormedrings} 
In the following, we give a couple of examples for hemiring-valued pseudonormed rings:
	
\begin{enumerate}
	\item Pseudonormed (normed) rings explained in \cite{ArnautovGlavatskyMikhalev1996}, are real-valued pseudonormed (normed) rings in the sense of Definition \ref{Hpseudonormedringsdef}.
	
	\item Let $R$ be a ring and $(G,\cdot,1,<)$ a totally ordered group. Assume that $0 \notin G$ and set $G_0= G \cup \{0\}$ and extend multiplication and inequality for $G_0$ with the following rules:
	
	\begin{enumerate}
		\item $0 \cdot g = g \cdot 0 = 0$, for all $g\in G$;
		\item $0 < g$, for all $g \in G$.
	\end{enumerate}
     
    It is said that $R$ is a $G_0$-valuation ring \cite[\S10]{Artin1957} if there is a function $|\cdot|: R \rightarrow G\cup \{0\}$ with the following properties:
     
     \begin{itemize}
     	\item $|r|=0$ if and only if $r = 0$, for all $r\in R$.
     	\item $|rs| = |r||s|$, for all $r,s\in R$.
     	\item $|r+s| \leq \max\{|r|,|s|\}$, for all $r,s \in R$.  
     \end{itemize}
 
     Note that $(G_0, \max, \cdot)$ is a (division) semiring (see Example 4.27 in \cite{Golan1999(b)}), and evidently, $R$ is a $G_0$-normed ring.  
	
	\item Let $\leq$ be a total ordering on $R$ and $(R,<)$ an ordered ring. Set \[|x| = \max\{x,-x\}, \qquad\forall x\in R.\] It is evident that $R$ is an $R$-normed ring. Note that if we define \[d(x,y) = |x-y|,\qquad \forall x,y \in R,\] then $(R,d)$ is an $R$-metric space.
\end{enumerate}

\end{examples}

Albert in Theorem 4 in \cite{Albert1947} proves that every finite-dimensional real algebra is a normed algebra (see also Proposition 1.1.7 in \cite{CabreraPalacios2014}). Inspired by his proof, we prove the following:

\begin{theorem}\label{finitedimpseudonormedthm}
Let $\leq$ be a total ordering on $H$ and $(H,+,\cdot,0,<)$ be an ordered	commutative hemiring. Also, let $F$ be a field and an $H$-pseudonormed ring. Then, every finite-dimensional $F$-algebra admits the structure of an $H$-pseudonormed ring.	
\end{theorem}

\begin{proof}
	Let $|\cdot|: F \rightarrow H$ be the $H$-pseudonorm of $F$.
	Let $R$ be a finite-dimensional $F$-algebra and $(e_i)_{i=1}^n$ be a basis for the $F$-vector space $R$. For any $a = \sum_{i=1}^{n} a_i e_i$, where $a_i \in F$, define $\Vert \cdot \Vert : R \rightarrow H$ as follows: \[\Vert a\Vert  = \sum_{i=1}^{n} |a_i|.\] It is obvious that $\Vert \cdot\Vert $ satisfies the following properties: 
	
	\begin{itemize}
		\item $\Vert a\Vert \geq 0$, $\Vert a\Vert  = 0$ if and only if $a = 0$, for all $a \in R$;
		\item $\Vert a-b\Vert  \leq \Vert a\Vert +\Vert b\Vert $ for all $a,b \in R$.
	\end{itemize} 
On the other hand, it is evident that \[e_i e_j = \sum_{k=1}^{n} \gamma_{ijk} e_k,\] for some $ \gamma_{ijk} \in F$. Set $M = \max\{|\gamma_{ijk}|\}_{i,j,k \in \mathbb N_n}$. Now, assume that \[a = \sum_{i=1}^{n} a_i e_i \text{~and~} b= \sum_{j=1}^{n} b_j e_j,\] where $a_i$s and $b_j$s are elements of $F$. Observe that \begin{align*} \Vert ab\Vert  = \left|\left| \sum_{1 \leq i,j \leq n}a_i b_j e_i e_j \right|\right| = \left|\left| \sum_{1 \leq i,j \leq n}a_i b_j \sum_{k=1}^{n} \gamma_{ijk} e_k \right|\right| \\ = \left|\left|  \sum_{k=1}^{n} \left(\sum_{1 \leq i,j \leq n}a_i b_j \gamma_{ijk} \right) e_k \right|\right| = \sum_{k=1}^{n} \left| \sum_{1 \leq i,j \leq n}a_i b_j \gamma_{ijk} \right| \\ \leq nM \sum_{1\leq i,j \leq n} |a_i||b_j| = nM \left(\sum_{i=1}^{n} |a_i|\right) \left(\sum_{j=1}^{n} |b_j|\right).
	\end{align*}
	This shows that $\Vert ab\Vert  \leq nM \Vert a\Vert \Vert b\Vert $. Since $H$ is commutative, it follows that \[nM \Vert ab\Vert  \leq \left(nM \Vert a\Vert \right) \left(nM ||b||\right).\] Therefore, if we define a new function $\Vert \cdot\Vert ': A \rightarrow H$ by \[\Vert a\Vert ' = nM \Vert a\Vert ,\] then $\Vert \cdot\Vert '$ provides a suitable $H$-pseudonorm for $R$ and this completes the proof.
\end{proof}

\begin{corollary}
Let $\leq$ be a total ordering on $F$ and $(F,<)$ be an ordered field. Any finite dimensional $F$-algebra admits the structure of an $F$-pseudonormed ring.
\end{corollary}

\begin{definition}\label{shrinkabledef}
Let $H$ be an ordered hemiring. We say $H$ is shrinkable if for the given positive elements $\alpha$ and $M$ in $H$, there are positive elements $\alpha_l$ and $\alpha_r$ such that \[M \alpha_r < \alpha \text{~and~} \alpha_l M < \alpha.\]
\end{definition}

\begin{definition}\label{densehemiringdef}
We say an ordered hemiring $(H,+,\cdot,<)$ is dense if its additive monoid is dense, i.e., for any $\alpha > 0$ in $H$, there are two elements $\beta > 0$ and $\gamma > 0$ in $H$ with $\beta + \gamma < \alpha$.  
\end{definition}

Let us recall that a semiring $S$ is a division semiring if any nonzero element of $S$ is multiplicatively invertible (see p. 52 in \cite{Golan1999(b)}).

\begin{proposition}\label{densedivisionsemiringshrinkable}
Let $D$ be an ordered division semiring. If $D$ is a dense semiring, then it is shrinkable. 
\end{proposition}

\begin{proof}
	Let $D$ be an ordered and a dense division semiring. Let a positive element $\alpha$ in $D$ be given. By Lemma 2.7 in \cite{Nasehpour2026}, there is positive element $\beta \in D$ with $\beta < \alpha$. Now, let $M$ be an arbitrary positive element in $D$. Since $D$ is a division semiring and $M$ is nonzero, it is invertible. Assume that \[M^{-1} M = M M^{-1} = 1\] and set $\alpha_r = M^{-1} \beta $ and $\alpha_l = \beta M^{-1}$. Clearly, \[M \alpha_r = \alpha_l M = \beta < \alpha.\] This completes the proof.
\end{proof}

\begin{examples}\label{shrinkablehemiringsex} In the following, we give some examples of dense division semirings which are, a fortiori, shrinkable by Proposition \ref{densedivisionsemiringshrinkable}.
\begin{enumerate}
	\item Let $(G,\cdot)$ be a totally ordered group. The semiring $(G_0, \max,\cdot)$ discussed in Examples \ref{examplesofpseudonormedrings} is shrinkable.
	
	\item By Theorem 2.13 in \cite{Nasehpour2026}, any DeMarr division ring is shrinkable. 
	
	\item Let $F$ be an ordered field (for example, $F=\mathbb Z(X)$ which is the ordered field of rational functions over $\mathbb Z$ \cite[Example 1.2]{Ovchinnikov2021}). It is clear that the non-negative elements $F^{\geq 0}$ of $F$ is a totally ordered semifield. Note that for any $\epsilon > 0$, we can find $\beta = \gamma = 2\epsilon/5$ satisfying the following property: \[\beta + \gamma = 4\epsilon/5 < \epsilon\] This shows that the monoid of non-negative elements of $F$,  denoted by $F^{\geq 0}$, is dense. This already means that $\mathbb Z(X)$ is shrinkable.
\end{enumerate}
\end{examples}

\begin{proposition}\label{pnumbersshrinkable}
	Let $p$ be a prime number and $\mathbb Z[1/p]$ the ring of all rational numbers of the form $m/p^n$, where $m \in \mathbb Z$ and $n \in \mathbb N$. Then, $\mathbb Z[1/p]$ is dense and shrinkable.
\end{proposition}

\begin{proof}
	Since $1/p$ which is smaller than 1 is an element of $R$, by Theorem 2.8 in \cite{Nasehpour2026}, the ring $R$ is dense. Now, let $\alpha = m_1/p^{n_1}$ and $M = m_2/p^{n_2}$ be given. It is obvious that for a suitable $n \in \mathbb N$, $\alpha_l = 1/p^n$ which is an element of $R$ satisfies the inequality $M \alpha_l < \alpha$. Since $R$ is commutative, $R$ is shrinkable and the proof is complete.
\end{proof}

\begin{lemma}\label{productboundedconvergent0}
Let $H$ be an ordered hemiring and $R$ an $H$-pseudonormed ring. Also, assume that $H$ is dense and shrinkable. If $(x_n)$ is a sequence in $R$ convergent to 0 and $(y_n)$ is a sequence in $R$ such that there is a positive element $M$ in $H$ with $\Vert y_n\Vert  \leq M$ for all $n \in \mathbb N$, then the sequences $(x_ny_n)$ and $(y_n x_n)$ are both convergent to $0 \in R$.
\end{lemma}

\begin{proof}
By assumption, there is a positive element $M$ in $H$ with $\Vert y_n\Vert  \leq M$. Let a positive element $\epsilon$ be given in $H$. Since $H$ is shrinkable, there are positive elements  $\epsilon_l$ and $\epsilon_r$ such that \[M \epsilon_r < \epsilon \text{~and~} \epsilon_l M < \epsilon.\] Since $(x_n)$ is convergent to 0, for $\epsilon_r > 0$, there is a natural number $N$ such that $n \geq N$ implies $\Vert x_n\Vert  < \epsilon_r$. Since $\Vert \cdot\Vert $ is an $H$-pseudonorm on $R$, we see that if $n \geq N$, we have \[\Vert x_n y_n\Vert  \leq \Vert x_n\Vert \Vert y_n\Vert  \leq M \epsilon_r < \epsilon,\] showing that $(x_ny_n)$ converges to $0\in R$. In a similar way, it is proved that $(y_nx_n)$ converges to $0\in R$ and the proof is complete.
\end{proof}

\begin{theorem}\label{multiplicationconvergentsequences}
Let $(S,\leq)$ be an ordered ring and $R$ an $S$-pseudonormed ring. Assume that $S$ is dense and shrinkable such that $(S,\leq)$ is a join-semilattice. Then, the following statements hold:

\begin{enumerate}
	\item If $(x_n)$ and $(y_n)$ are sequences in $R$ convergent to $a$ and $b$ in $R$, respectively, then $(x_ny_n)$ is convergent to $ab$.
	
	\item The set $\Conv(R)$ is a ring and the function \[\varphi: \Conv(R) \rightarrow R \text{~defined~by~} (x_n) \mapsto \lim_{n \to +\infty} x_n\] is a ring epimorphism with $\ker(\varphi) = \Zero(R)$, where by $\Zero(R)$, we mean the set of all sequences in $R$ convergent to 0.
	
	\item Let $R$ be a commutative ring with 1. The ideal $\Zero(R)$ is prime (maximal) if and only if $R$ is an integral domain (a field). 
\end{enumerate}

\end{theorem}

\begin{proof}
(1): Let for the moment $\lim_{n \to +\infty} y_n = b \neq0$. Let a positive $\epsilon$ in $S$ be given. Since $S$ is dense, by Theorem 1.8 in \cite{Nasehpour2026}, we can find two positive elements $\beta$ and $\gamma$ such that $\beta + \gamma < \epsilon$. Since $(x_n)$ is a convergent sequence, by Proposition 4.5 in \cite{Nasehpour2026}, we can find a positive element $s_1 \in S$ such that \[\Vert x_n\Vert  \leq s_1, \qquad\forall~ n \in \mathbb N.\] Since $S$ is shrinkable, for $\beta > 0$ and $s_1 > 0$, and also for $\gamma > 0$ and $\Vert b\Vert  > 0$, we can find positive elements $M_l$ and $K_r$ in $S$ such that \[s_1K_r < \beta  \text{~and~} M_l \Vert b\Vert  < \gamma.\] Also, for $K_r > 0$, we can find a natural number $N_1$ such that $n \geq N_1$ implies $\Vert y_n - b\Vert  < K_r$. Similarly, we can find a natural number $N_2$ such that $n \geq N_2$ implies $\Vert x_n - a\Vert  < M_l$. Now, set $N = \max\{N_1,N_2\}$ and assume that $n \geq N$. Observe that \begin{align*}
\Vert x_n y_n - ab\Vert  = \Vert x_n y_n - x_n b + x_n b - ab\Vert  \\ \leq  \Vert x_n (y_n - b)\Vert  + \Vert (x_n - a)b\Vert  \\ \leq \Vert x_n\Vert \cdot\Vert y_n - b\Vert  + \Vert x_n - a\Vert \cdot\Vert b\Vert  \\ \leq s_1 K_r + M_l \Vert b\Vert  \leq \beta + \gamma < \epsilon.
\end{align*} This proves that $(x_ny_n)$ is convergent to $ab$. On the other hand, if $\lim_{n \to +\infty} y_n = 0$, then by Lemma \ref{productboundedconvergent0} $(x_ny_n)$ is convergent to $ab = 0$.

(2): In view of Theorem 4.6 in \cite{Nasehpour2026} and (1), $\varphi$ is a ring epimorphism and its kernel is $\Zero(R)$.

(3) Since by (2) we have $\Conv(R) / \Zero(R) \cong R$, the result is obviously obtained and the proof is complete.
\end{proof}

\begin{corollary}
Let $\leq$ be a total ordering on $R$ and $(R,<)$ be an ordered dense commutative ring with identity. If $R$ is shrinkable, then $\Conv(R)$ is a ring and $\Zero(R)$ is a prime ideal with $\Conv(R)/\Zero(R) \cong R$.
\end{corollary}

\section{Infinite series in $M$-normed groups}\label{sec:infiniteseries}

Let us recall that a group $(G,+)$ is an $M$-normed group (cf. Definition 4.1 in \cite{Nasehpour2026}) if $M$ is a dense ordered unital magma and there is a function $\Vert \cdot \Vert: G \rightarrow M$ with the following properties:

\begin{enumerate}
	\item $\Vert g \Vert \geq 0$, and $\Vert g \Vert = 0$ if and only if $g = 0$, for all $g\in G$,
	\item $\Vert g-h \Vert \leq \Vert g \Vert + \Vert h \Vert$, for all $g,h \in G$.
\end{enumerate}

\begin{proposition}\label{convergencesumofseries}
	Let $M$ be a dense ordered unital magma and $G$ an $M$-normed abelian group. Then, the following statements hold:
	
	\begin{enumerate}
		\item If $\sum^{+\infty}_{n=1} x_n$ is convergent, then $(x_n)$ converges to zero.
		\item If $\sum^{+\infty}_{n=1} x_n$ is convergent, then for any positive $\epsilon$ in $M$, there is a natural number $N$ such that $n \geq m \geq N$ implies $\Vert \sum_{i=m+1}^{n} x_i\Vert < \epsilon$.  
		\item If $\sum^{+\infty}_{n=1} x_n$ and $\sum^{+\infty}_{n=1} y_n$ are convergent series in $G$, then \[\sum^{+\infty}_{n=1} (x_n+y_n) = \sum^{+\infty}_{n=1} x_n + \sum^{+\infty}_{n=1} y_n.\]
	\end{enumerate}
\end{proposition}

\begin{proof}
(1): Set $s_n = x_1 + x_2 + \dots + x_n$. If $(s_n)$ is convergent to $s \in G$, then it is straightforward (Proposition 3.10 in \cite{Nasehpour2026}) to see that $(s_{n+1})$ is also convergent to $s$. Since $G$ is abelian, we have $x_{n+1} = s_{n+1} - s_n$. So, we see that by Theorem 4.6 in \cite{Nasehpour2026}, we have \[\lim_{n \to +\infty} x_{n+1} = \lim_{n \to +\infty} s_{n+1} - \lim_{n \to +\infty} s_{n} = s - s = 0.\] Thus by Proposition 3.10 in \cite{Nasehpour2026}, $(x_n)$ is convergent to zero.

(2): Since $M$ is dense, by Theorem 4.6 in \cite{Nasehpour2026}, any convergent sequence is a Cauchy sequence. Therefore, for the given $\epsilon > 0$, there is a natural number $N$ such that $n \geq m \geq N$ implies \[\left|\left|\sum_{i=m+1}^{n} x_i\right|\right| = \Vert s_n - s_m \Vert  < \epsilon.\] 

(3): It evidently holds by Theorem 4.6 in \cite{Nasehpour2026}.
\end{proof}

\begin{theorem}
Let $\leq$ be a total ordering on $R$ and $(R,<)$ an ordered ring and $(x_n)$ be a strictly decreasing and positive sequence convergent to zero. Then, $s_n = \sum_{i=1}^{n} (-1)^{i+1} x_i$ is a Cauchy sequence.
\end{theorem}

\begin{proof}
	Let $n > m$ and observe that \begin{align*}
		|s_n - s_m| = \left|\sum_{i=1}^{n} (-1)^{i+1} x_i - \sum_{i=1}^{m} (-1)^{i+1} x_i \right| \\ = |x_{m+1} - x_{m+2} + \dots + (-1)^{n-m+1}x_n|
	\end{align*} Since $(x_n)$ is a strictly decreasing and positive sequence, we see that \[|s_n - s_m| = \begin{cases}
		x_{m+1} - \sum_{i=m+2}^{n-1} (x_{i} - x_{i+1}) \leq x_{m+1} < x_m, & \text{if~} n-m+1 \text{~is~even,}\\
		x_{m+1} - \sum_{i=m+2}^{n-2} (x_{i} - x_{i+1}) -x_n \leq x_{m+1} < x_m, & \text{if~} n-m+1 \text{~is~odd.}
	\end{cases}\] Now, since $\lim_{n \to +\infty} x_n = 0$, $(s_n)$ is a Cauchy sequence and the proof is complete.
\end{proof}

\begin{definition}
Let $\leq$ be a total ordering and $(R,<)$ an ordered ring. We define $(R,<)$ to be Cauchy complete if each Cauchy sequence in $R$ is convergent to an element in $R$.
\end{definition}

\begin{corollary}
Let $\leq$ be a total ordering and $(R,<)$ an ordered ring. If $(R,<)$ is a Cauchy complete ring and $(x_n)$ is a strictly decreasing and positive sequence convergent to zero, then \[s_n = \sum_{i=1}^{n} (-1)^{i+1} x_i\] is a convergent sequence.
\end{corollary}

\begin{theorem}\label{DenseCauchycompletethm} 
Let $\leq$ be a total ordering on $R$ and $(R,<)$ a dense and Cauchy complete ordered ring. Assume that $(x_n)$, $(y_n)$, and $(z_n)$ are sequences in $R$ and there is a natural number $N_1$ such that $n \geq N_1$ implies \[x_n \leq y_n \leq z_n.\] If $\sum_{n=1}^{+\infty} x_n$ and $\sum_{n=1}^{+\infty} z_n$ are convergent, then so is the series $\sum_{n=1}^{+\infty} y_n$.  	
\end{theorem}

\begin{proof}
	Since $\sum_{n=1}^{+\infty} x_n$ and $\sum_{n=1}^{+\infty} z_n$ are convergent, by Theorem 3.15 in \cite{Nasehpour2026}, they are Cauchy sequences. Therefore, for the given $\epsilon > 0$, we can find a natural number $N_2$ such that $m,n \geq N_2$ implies \[\left|\sum_{k=m+1}^{n} x_k\right| < \epsilon \text{~and~} \left|\sum_{k=m+1}^{n} z_k\right| < \epsilon.\] Let $N = \max\{N_1,N_2\}$. Observe that for $n \geq N$, we have $x_n \leq y_n \leq z_n$ which implies that \[\left|\sum_{k=m+1}^{n} y_k\right| < \epsilon.\] Therefore, the sequence $\big(\sum_{k=1}^{n} y_k\big)_{n\in \mathbb N}$ is Cauchy and so, convergent since $R$ is Cauchy complete. This completes the proof.
\end{proof}

\begin{remark}
	The question arises if there is any ring other than the field of real numbers satisfying the conditions of Theorem \ref{DenseCauchycompletethm}. A famous fact in the theory of ordered fields states that ``a totally ordered field is Dedekind complete (i.e. isomorphic to the field of real numbers) if and only if it is Cauchy complete and Archimedean (see Theorem 2.5 in \cite{Ovchinnikov2021}). Therefore, any Cauchy complete field which is not Archimedean is an ordered field other than the field of real numbers. For example, let $F=\mathbb Z(X)$ be the ordered field of rational functions over $\mathbb Z$ (see Example 1.2 in \cite{Ovchinnikov2021}). Since $F$ is not Archimedean, its completion $\widetilde{F}$ is also not Archimedean. Therefore, $\widetilde{F}$ is an example of a Cauchy complete field which is not isomorphic to $\mathbb R$ (see Example 2.3 in \cite{Ovchinnikov2021}). Obviously, $\widetilde{F}$ is a dense ring.
\end{remark}

Cauchy in his famous book entitled ``Cours d'analyse'' proves that whenever each term of the series $\sum_{i=0}^{+\infty} u_i$ is non-negative and smaller than the one preceding it, that series and $\sum_{i=0}^{+\infty} 2^i u_{2^i - 1}$ are either both convergent or both divergent (see Theorem III on p. 92 in \cite{BradleySandifer2009}). This is known as the Cauchy Condensation Test in many resources. For example, see Theorem 3.27 on p. 61 and its name on p. 339 in Rudin's book \cite{Rudin1976}. There is nothing special for the number 2 in the condensed series $\sum_{i=0}^{+\infty} 2^i u_{2^i - 1}$. For a detailed discussion of Cauchy Condensation Test see pages 120--122 in \cite{Knopp1951}.

\begin{theorem}[Cauchy Condensation Test]\label{Cauchycondensationtest}
Let $\leq$ be a total ordering on $R$ and $(R,<)$ a dense and Cauchy complete ordered ring with $1$. Assume that $(x_n)$ is a positive and decreasing sequence in $R$. Then, $\sum_{i=1}^{+\infty} x_i$ is convergent in $R$ if and only if $\sum_{i=0}^{+\infty} 2^i x_{2^i}$ is convergent in $R$.
\end{theorem}

\begin{proof}
$(\Rightarrow)$: Since $(x_n)$ is a decreasing sequence, for all $n \in \mathbb N$, we have \begin{align} 2^n x_{2^n} = 2^{n-1} (x_{2^n} + x_{2^n}) \leq 2\sum_{i=2^{n-1}+1}^{2^n} x_i. \label{Cauchycondensationformula1} \end{align} Now, suppose that $k \leq l$. Take $m$ and $n$ be positive integers such that \[m \leq 2^{k-1}+1 \text{~and~} 2^l \leq n.\] Then, by using the inequality in (\ref{Cauchycondensationformula1}), we have \begin{align} \sum_{i=k}^{l} 2^i x_{2^i} \leq 2\sum_{i=2^{k-1}+1}^{2^l} x_i \leq 2\sum_{i=m}^{n} x_i. \label{Cauchycondensationformula2} \end{align} Since $\sum_{i=1}^{+\infty} x_i$ is convergent and $R$ is dense, by Theorem 3.15 in \cite{Nasehpour2026}, $\sum_{i=1}^{+\infty} x_i$ is a Cauchy sequence. Let a positive element $\epsilon$ in $R$ be given. Since $R$ is dense, we can find positive elements $\beta$ and $\gamma$ in $R$ with $\beta + \gamma < \epsilon$. For $\beta$ $(\gamma)$, we can find a natural number $N_{\beta}$ $(N_{\gamma})$ such that \[N_{\beta} \leq m \leq n~(N_{\gamma} \leq m \leq n)\] implies $\sum_{i=m}^{n} x_i < \beta$ $(\sum_{i=m}^{n} x_i < \gamma)$. Now, set $N_1 = \max\{N_{\beta}, N_{\gamma}\}$. Then, for $N_1 \leq m \leq n$, we have \begin{align} 2\sum_{i=m}^{n} x_i = \sum_{i=m}^{n} x_i + \sum_{i=m}^{n} x_i \leq \beta + \gamma < \epsilon. \label{Cauchycondensationformula3} \end{align} Imagine $k$ is the smallest natural number with $m \leq 2^{k-1}+1$. If $k \leq l$, then $2^{k-1}+1 \leq 2^l$. Therefore, for a suitable natural number $N_2$, if $N_2 \leq k \leq l$, then in view of (\ref{Cauchycondensationformula2}) and (\ref{Cauchycondensationformula3}), we have \[\sum_{i=k}^{l}2^i x_{2^i} \leq 2\sum_{i=2^{k-1}+1}^{2^l} x_i \leq \beta + \gamma < \epsilon\] showing that  $\sum_{i=0}^{+\infty} 2^i x_{2^i}$ is a Cauchy sequence. Since $(R,<)$ is a Cauchy complete ordered ring, $\sum_{i=0}^{+\infty} 2^i x_{2^i}$ is convergent.
	
$(\Leftarrow)$: Since $(x_n)$ is a decreasing sequence, for all $n \in \mathbb N$, we have \begin{align} x_{2^n} + x_{{2^n}+1} + \dots + x_{2^{n+1} - 1} \leq 2^n x_{2^n}. \label{Cauchycondensationformula4} \end{align} Now, let $m \leq n$ be positive integers. Imagine $k$ is the greatest non-negative integer satisfying $2^k \leq m$ and $l$ is the smallest non-negative integer with $n \leq 2^{l+1} - 1$. Observe that \begin{align}
		\sum_{i=m}^{n} x_i \leq \sum_{i=2^k}^{2^{l+1}-1} x_i \leq \sum_{i=k}^{l} 2^i x_{2^i}.\label{Cauchycondensationformula5} \end{align} Since $\sum_{i=0}^{+\infty} 2^i x_{2^i}$ is convergent and $R$ is dense, by Theorem 3.15 in \cite{Nasehpour2026}, $\sum_{i=0}^{+\infty} 2^i x_{2^i}$ is a Cauchy sequence. Therefore, for the given positive $\epsilon$, there is a natural number $N_1$ such that for arbitrary $k,l$ with $N_1 \leq k \leq l$ we have $\sum_{i=k}^{l} 2^i x_{2^i} < \epsilon$. Now, in view of the inequality in (\ref{Cauchycondensationformula5}), we see that there is a natural number $N_2$ such that $N_2 \leq m \leq n$ implies \[\sum_{i=m}^{n} x_i < \epsilon\] showing that  $\sum_{i=1}^{+\infty} x_i$ is a Cauchy sequence. Since $(R,<)$ is a Cauchy complete ordered ring, $\sum_{i=1}^{+\infty} x_i$ is convergent and the proof is complete.
\end{proof}

\begin{corollary}\label{Cauchycondensationtestcor}
	Let $\leq$ be a total ordering on $F$, and $(F,<)$ a Cauchy complete totally ordered field and $(x_n)$ a positive and decreasing sequence in $F$. Then, $\sum_{i=1}^{+\infty} x_i$ is convergent in $F$ if and only if $\sum_{i=0}^{+\infty} 2^i x_{2^i}$ is convergent in $F$.
\end{corollary}

\begin{definition}
Let $(M,+,0,\leq)$ be a dense monoid and $G$ an $M$-normed abelian group.

\begin{enumerate}
	\item By definition, the $M$-normed group $G$ is Cauchy complete if any Cauchy sequence in $G$ is convergent.
	
	\item Let $x_n \in G$ for each $n \in \mathbb N$. It is said that $\sum_{i=1}^{+\infty} x_i$ is an absolutely convergent series if $\sum_{i=1}^{+\infty} \Vert x_i\Vert $ is convergent.  
\end{enumerate}
\end{definition}

\begin{theorem}\label{absolutelyconvergent}
Let $M$ be a totally ordered abelian group. Also, let $M$ be a dense group, $G$ an $M$-normed abelian group, and $x_n \in G$ for each $n \in \mathbb N$. If $G$ is Cauchy complete and $\sum_{i=1}^{+\infty} x_i$ is a absolutely convergent series then the series $\sum_{i=1}^{+\infty} x_i$ is convergent in $G$. 
\end{theorem}

\begin{proof}
Since $M$ is dense, by Theorem 3.15 in \cite{Nasehpour2026}, any convergent sequence in $G$ is a Cauchy sequence. Since $\sum_{i=1}^{+\infty} \Vert x_i\Vert $ is convergent in $M$, in view of Proposition 3.4 in \cite{Nasehpour2026}, it is a Cauchy sequence in the $M$-normed group $M$. Therefore, for the given positive $\epsilon$, we can find a natural number $N$ such that $n \geq m \geq N$ implies the following: \[\Vert x_{m+1}\Vert  + \dots + \Vert x_n\Vert  < \epsilon.\] On the other hand, \[\Vert x_{m+1} + \dots + x_n\Vert  \leq \Vert x_{m+1}\Vert  + \dots + \Vert x_n\Vert  < \epsilon.\] So, the series $\sum_{i=1}^{+\infty} x_i$ is a Cauchy sequence. Since the $M$-normed group $G$ is Cauchy complete, $\sum_{i=1}^{+\infty} x_i$ is convergent in $G$ and the proof is complete. 
\end{proof}

\section{The geometric series in dense and shrinkable commutative rings}\label{sec:geometricseries}

\begin{theorem}\label{Geometricseries}
Let $\leq$ be a total ordering on $R$ and $(R,<)$ an ordered commutative ring with 1 such that $R$ is a dense and shrinkable ring. Assume that $r \neq 1$. Then, the following statements hold:

\begin{enumerate}
	\item If the sequence $(r^n)$ is convergent, then it converges to zero.
	\item The geometric series $\sum_{n=0}^{+\infty} r^n$ is convergent in $R$ if and only if the sequence $(r^n)$ converges to zero and $1-r$ is invertible in $R$.
\end{enumerate}
\end{theorem}

\begin{proof}
(1): If $r = 0$, then $(r^n)$ converges to zero. Now, assume that $r$ is nonzero and $(r^n)$ converges to $l\in R$. Set $x_n = r^n$. Evidently, $x_{n+1} = r x_n$ and $x_{n+1}$ converges to $l$ also. Observe that by Theorem \ref{multiplicationconvergentsequences}, we have \[l = \lim_{n \to +\infty} x_{n+1} = \lim_{n \to +\infty} rx_n = r \lim_{n \to +\infty} x_n = rl.\] Since $(R,<)$ is an ordered ring, $R$ is an integral domain. Now, from $l(r-1) = 0$ and $r \neq 1$, we deduce that $l = 0$.

(2): Assume that $\sum_{i=0}^{n} r^i$ is convergent to $\alpha \in R$. By Proposition \ref{convergencesumofseries}, $(r^n)$ converges to zero. Now, set $s_n = \sum_{i=0}^{n} r^i$. By Theorem \ref{multiplicationconvergentsequences}, we have \[(1-r)\alpha = (1-r)\lim_{n \to +\infty} s_n = \lim_{n \to +\infty}(1-r^{n+1}) = 1\] showing that $1-r$ is invertible in $R$. 

Conversely, since $1-r$ is invertible, we have \[1+r+\dots+r^n = \displaystyle \frac{1-r^{n+1}}{1-r}.\] Now, if $(r^n)$ converges to zero, we have $\lim_{n \to +\infty} r^{n+1} = 0$ and so, we obtain that \begin{align*}
\sum_{n=0}^{+\infty} r^n = \lim_{n \to +\infty} (1+r+\dots+r^n) = \\ \lim_{n \to +\infty} \displaystyle \frac{1-r^{n+1}}{1-r} = \displaystyle \frac{1}{1-r}\lim_{n \to +\infty} (1-r^{n+1}) = \displaystyle \frac{1}{1-r}.
\end{align*} This completes the proof.
\end{proof}

Let us recall that if each of the real numbers $x_i$s are greater than $-1$ and either all are positive or negative, then \[\prod_{i=1}^{n} (1+x_i) \geq 1 + \sum_{i=1}^{n} x_i \qquad\text{(See p. 35 in \cite{Mitrinovic1970}}).\] In the following, we prove this for (not necessarily commutative) ordered semirings:

\begin{theorem}[Generalization of Bernoulli's Inequality for Ordered Semirings] \label{Mitrinovicsemirings}
	Let $(S,\leq)$ be an ordered semiring and $(x_i)_{i=1}^{n}$ be a sequence of $n$ elements in $S$ such that $x_i\geq 0$ and $1+x_i \geq 0$, for each $1 \leq i \leq n$. Then, the following inequality holds: \[\prod_{i=1}^{n} (1+x_i) \geq 1 + \sum_{i=1}^{n} x_i.\]
\end{theorem}

\begin{proof}
	The proof is by induction on $n$. It is evident that the inequality holds for $n=1$. Now, let the inequality hold for $n = k$, i.e. \begin{align}	
		\prod_{i=1}^{k} (1+x_i) \geq 1 + \sum_{i=1}^{k} x_i. \label{Mitrinovic1} \end{align}
	
	Now, multiply the both sides of the inequality (\ref{Mitrinovic1}) by $1+x_{i+1}\geq 0$ (from the left side) and observe that \begin{align*}
		(1+x_{k+1})\left(\prod_{i=1}^{k} (1+x_i)\right) \geq (1+x_{k+1})\left(1+\sum_{i=1}^{k} x_i\right) \geq \\ (1+\sum_{i=1}^{k} x_i) + x_{k+1} + \sum_{i=1}^{k} x_{k+1} x_i  \geq 1+\sum_{i=1}^{k+1} x_i.
	\end{align*} This completes the proof.
\end{proof}

\begin{corollary}[Generalization of Bernoulli's Inequality for Ordered Rings] \label{Mitrinovicrings}
	Let $(R,\leq)$ be an ordered ring and $(x_i)_{i=1}^{n}$ be a sequence of $n$ elements in $R$ such that either $x_i\geq 0$ for each $1 \leq i \leq n$, or $x_i\leq 0$ for each $1 \leq i \leq n$. Also, suppose that $1+x_i \geq 0$, for each $1 \leq i \leq n$. Then, the following inequality holds: \[\prod_{i=1}^{n} (1+x_i) \geq 1 + \sum_{i=1}^{n} x_i.\]
\end{corollary}

\begin{proof}
The only case that it must be shown is that from $x_{i+1} \leq 0$ and $x_i \leq 0$, we obtain that $0 \leq -x_{i+1}$ and $0 \leq - x_i$ which implies that \[0 \leq (- x_{i+1})(- x_i) = x_{i+1}x_i.\] The rest of the proof is similar to the proof of Theorem \ref{Mitrinovicsemirings}.
\end{proof}

\begin{corollary}\label{Bernoullisinequality}
	Let $(S,\leq)$ be an ordered semiring such that $x \geq0$ and $1+x \geq 0$. Then, for each $n \in \mathbb N$, we have \[(1+x)^n \geq 1+nx \qquad\text{(Bernoulli's inequality)}.\]
\end{corollary}

Let us recall that a totally ordered field $F$ is Archimedean if any positive element $x$ and any element $y$ in $F$, there is a natural number $n$ such that $nx > y$ \cite[Definition 1.8]{Ovchinnikov2021}.

\begin{proposition}\label{limitpowersequence}
	Let $(F, <)$ be a totally ordered and Archimedean field and $r \in F$ with $- 1 < r < 1$. Then, \[\lim_{n \to +\infty} r^n = 0.\] 	
\end{proposition}

\begin{proof}
	Without loss of generality, we may assume that $0 < r < 1$. Set $x = (1/r) - 1$ and let $\epsilon > 0$ be given. By definition, there is a natural number $N$ such that $N(x\epsilon) > 1$. Now, let $n \geq N$. In view of Corollary \ref{Mitrinovicrings}, we have \[r^n = \frac{1}{(1+x)^n} \leq \frac{1}{1+nx} \leq \frac{1}{Nx} < \epsilon.\]	Hence, $\lim_{n \to +\infty} r^n = 0$, as required.
\end{proof}

In view of Theorem \ref{Geometricseries} and Proposition \ref{limitpowersequence}, we have the following:

\begin{corollary}(\cite[Exercise 6.5]{Ovchinnikov2021}) \label{GeometricseriesOvchinnikov2021}
Let $(F,<)$ be a totally ordered and Archimedean field. Assume that $r \in F$ with $-1 < r < 1$. Then, \[\sum_{n=0}^{+\infty} r^n = \displaystyle \frac{1}{1-r}.\]
\end{corollary}

The following is a counterpart of a version of ``ratio test'' in mathematical analysis:

\begin{theorem}[Ratio Test]\label{ratiotestcounterpart}
Let $\leq$ be a total ordering on $R$ and $(R,<)$ an ordered commutative ring with 1 such that $R$ is dense and shrinkable. Assume that $1-r$ is invertible and $(r^n)$ converges to zero in $R$. Also, let $(G,+)$ be a Cauchy complete $R$-normed abelian group and $(x_n)_{n=0}^{+\infty}$ be a sequence in $G$ such that $\Vert x_{n+1}\Vert  \leq r\Vert x_n\Vert $ for all $n$. Then, $\sum_{n=0}^{+\infty} x_n$ is convergent in $G$.  
\end{theorem}

\begin{proof}
By Theorem \ref{Geometricseries}, $\sum_{n=0}^{+\infty} r^n$ is convergent. This implies that the series \[\sum_{n=0}^{+\infty} (r^n\Vert x_0\Vert )\] is also convergent. On the other hand, $R$ is dense. So, by Theorem 3.15 in \cite{Nasehpour2026}, \[\sum_{n=0}^{+\infty} (r^n \Vert x_0\Vert )\] is a Cauchy sequence. It is evident that \[0 \leq \Vert x_n\Vert  \leq r^n \Vert x_0\Vert , \qquad \forall~n \in \mathbb N.\] Similar to the proof of Theorem \ref{DenseCauchycompletethm}, we can see that $\sum_{n=0}^{+\infty} \Vert x_n\Vert $ is convergent. Also, similar to the proof of Theorem \ref{absolutelyconvergent}, it is proved that $\sum_{n=0}^{+\infty} x_n$ is convergent. This completes the proof.
\end{proof}

\section*{Acknowledgments} The author is grateful to Prof. Winfried Bruns for his encouragements and wishes to thank Prof. Henk Koppelaar for introducing the classical book \cite{Knopp1951}.

\end{document}